\documentclass{amsart}

\usepackage{amssymb,amsmath,color,ytableau,graphicx,float}
\usepackage{amsthm, url}
\usepackage{url}
\usepackage{tikz}
\usepackage{hyperref}
\usepackage[enableskew]{youngtab}

\definecolor{red}{rgb}{1,0,0}
\definecolor{blue}{rgb}{.2,.2,.8}

\def\mex{\textrm{mex}}

\newtheorem{theorem}{Theorem}[section]
\newtheorem{corollary}[theorem]{Corollary}
\newtheorem{proposition}{Proposition}
\newtheorem{conjecture}{Conjecture}

\theoremstyle{definition}

\newtheorem{remark}{Remark}

\begin{document}
\allowdisplaybreaks

\title[Truncated Theta Series Related to the JTP Identity]{Truncated Theta Series Related to the Jacobi Triple Product Identity}

\author{Cristina Ballantine}\address{Department of Mathematics and Computer Science\\ College of the Holy Cross \\ Worcester, MA 01610, USA \\} 
\email{cballant@holycross.edu} 
\author{Brooke Feigon} \thanks{B.F. partially supported by Simons Foundation Collaboration Grant 635835}\address{Department of Mathematics \\ The City College of New York \\ New York, NY 10031,  USA \\} \email{bfeigon@ccny.cuny.edu} 
\begin{abstract} The work of Andrews and Merca on the truncated Euler's pentagonal number theorem led to a resurgence in research on truncated theta series identities. In particular, Yee proved a truncated version of the Jacobi Triple Product (JTP) identity. Recently, Merca conjectured a stronger form of the truncated JTP identity. In this article we prove the first three cases of the conjecture and several related truncated identities. We  prove combinatorially an identity related to the JTP identity which in particular cases reduces to identities conjectured by Merca and proved analytically by
 Krattenthaler, Merca and Radu. Moreover, we introduce a new combinatorial interpretation for the number of distinct $5$-regular partitions of $n$. 
\end{abstract}



\maketitle

\section{Introduction}
A partition $\lambda$ of $n$ is a non-increasing sequence $\lambda= (\lambda_1, \lambda_2, \ldots, \lambda_\ell)$ of positive integers that add up to $n$. We refer to the integers $\lambda_i$ as the parts of $\lambda$.  As usual, we denote by $p(n)$  the number of  partitions of $n$. Since the empty partition $\emptyset$ is the only partition of $0$, we have  $p(0)=1$.   For a thorough introduction to the theory of partitions we refer the reader to  \cite{A98, AE}. 

We   use the Pochhammer symbol: for $n \in \mathbb N_0 \cup \{\infty\}$  
$$(a;q)_n := \prod_{j=0}^{n-1} (1-a q^j).$$

The generating function for the sequence $p(n)$ is given by $$\sum_{n=0}^\infty p(n)q^n=\prod_{i=1}^\infty \frac{1}{1-q^i}=\frac{1}{(q;q)_\infty}.$$
Throughout this article, we assume $q$ is a complex number with $|q|<1$ so that all products and series converge absolutely. 

The earliest theta series identity is likely Euler's pentagonal number theorem \cite[Eq. (1.3.1)]{A98}: \begin{equation}\label{epnt} \sum_{j=-\infty}^\infty (-1)^jq^{j(3j-1)/2}=(q;q)_\infty. \end{equation} Multiplying both sides in \eqref{epnt} by $1/(q;q)_\infty$, leads to the recurrence relation $$p(n)=  \sum_{j=1}^\infty (-1)^{j-1} p(n-j(3j\pm 1)/2)$$
for $n\geq 1$.

In 1951, Shanks \cite{S1} considered a truncation of the  theta series occurring in \eqref{epnt} and proved  that \begin{equation} \label{S1-trunc} \sum_{j=-k}^k(-1)^jq^{j(3j-1)/2}=\sum_{j=0}^k\frac{(-1)^j(q;q)_kq^{\binom{j+1}{2}+jk}}{(q;q)_j}.
\end{equation}
With few exceptions \cite{S2, A86} the subject was dormant for more than half a century. Then, in 2012, Andrews and Merca \cite{AM} obtained a new truncation of Euler's pentagonal number theorem. They showed that \begin{equation}\label{epnt-trunc}\frac{1}{(q;q)_\infty}\sum_{j=0}^{k-1}(-1)^iq^{j(3j+1)/2}(1-q^{2j+1})=1+(-1)^{k-1}\sum_{n=1}^\infty \frac{q^{\binom{k}{2}+(k+1)n}}{(q;q)_n}
\begin{bmatrix} n-1\\k-1\end{bmatrix},
\end{equation}
where  
$$
\begin{bmatrix}
n\\k
\end{bmatrix} 
=
\begin{cases}
\dfrac{(q;q)_n}{(q;q)_k(q;q)_{n-k}} &  \text{if $0\leqslant k\leqslant n$},\\ \ \\ 
0 &\text{otherwise.}
\end{cases}
$$ As explained in \cite{AM}, combinatorially,  \eqref{epnt-trunc} means that $$(-1)^{k-1}\sum_{j=-(k-1)}^{k}(-1)^jp(n-j(3j-1)/2)=M_k(n),$$ where $M_k(n)$ is the number of partitions of $n$ in which $k$ is the smallest integer that is not a part of the partition and there are more parts greater than $k$ than there are less than $k$. 
 
Work of Merca \cite[Corollary 17]{M4} shows that Shank's truncation \eqref{S1-trunc}  leads to  the following inequality which holds for all $n, k\geq 1$: \begin{equation}\label{S1-ineq} (-1)^k \sum_{j=-k}^{k}(-1)^jp(n-j(3j-1)/2)\geq 0.  \end{equation} Xia and Zhao \cite{XZ} showed that the left hand side of \eqref{S1-ineq} equals the number of partitions of $n$  in which every integer less than or equal to $k$ appears as a part at least once and the first part larger than $k$ appears at least $k+1$ times.
 
 The work of Andrews and Merca opened up the study of truncated theta series and linear partition inequalities and many articles followed. For example, Guo and Zeng \cite{GZ} obtained truncations of Gauss' identities \cite[(2.2.12), (2.2.13)]{A98}, and Yee \cite{Y} and also Wang and Yee \cite{WY} found a truncation of the Jacobi triple product identity \begin{equation}\label{JTP_RS} \frac{1}{(q^S;q^R)_\infty(q^{R-S};q^R)_\infty(q^R;q^R)_\infty}\sum_{n=-\infty}^\infty(-1)^nq^{Rn(n-1)/2+Sn}=1, \end{equation} where $R$ and $S$ are positive integers with $1\leq S< R/2$. 
 Wang and Yee show explicitly that the truncations 
 \begin{align*}\frac{1}{(q^S;q^R)_\infty(q^{R-S};q^R)_\infty(q^R;q^R)_\infty}\sum_{n=-(m-1)}^m(-1)^nq^{Rn(n-1)/2+Sn}  \end{align*} and  \begin{align*} \frac{1}{(q^S;q^R)_\infty(q^{R-S};q^R)_\infty(q^R;q^R)_\infty}\sum_{n=-m}^m(-1)^nq^{Rn(n-1)/2+Sn},
 \end{align*}
 when expanded as $q$-series, have non-negative coefficients. 
 
 Motivated by these results, Merca \cite{M1} investigated truncations of Jacobi's triple product written in the form \begin{equation}\label{JTP_RS_M} \frac{1}{(q^S;q^R)_\infty(q^{R-S};q^R)_\infty}\sum_{n=-\infty}^\infty(-1)^nq^{Rn(n-1)/2+Sn}-(q^R;q^R)_\infty=0 \end{equation}
 and made a stronger conjecture.
 
 \begin{conjecture}{\cite[Conjecture 4.3]{M1}} \label{Conj 4.3} For $1\leq S<R$, $k\geq 1$, the theta series  \begin{equation}\label{conj 4.3} \frac{(-1)^k}{(q^S;q^R)_\infty(q^{R-S};q^R)_\infty}\sum_{j=k}^\infty(-1)^jq^{Rj(j+1)/2-Sj}\big(1-q^{S(2j+1)}\big) \end{equation} has non-negative coefficients. 
 \end{conjecture}
 
  We define   \begin{align}\label{trunc-equiv} \sum_{n=0}^\infty a_{S,R,k}(n)q^n:=\frac{(-1)^{k-1}}{(q^S;q^R)_\infty(q^{R-S};q^R)_\infty}& \notag \sum_{j=0}^{k-1} (-1)^jq^{j(j+1)R/2-jS}(1-q^{(2j+1)S})\\ & -(-1)^{k-1}(q^R;q^R)_\infty, \end{align} and we write  $a_k(n)$ for $a_{S,R,k}(n)$ when $S$ and $R$ are clear from the context. 
  As mentioned in \cite{M1}, Conjecture \ref{Conj 4.3} is equivalent to showing $a_{S,R,k}(n)\geq 0$ for all $n\geq 0$, $k\geq 1$, $1\leq S<R$. 
 
 In this article we prove Conjecture \ref{Conj 4.3} for $k\in \{1,2,3\}$. 
 
 \begin{theorem}\label{main} If $k\in \{1,2,3\}$ and $1\leq S<R$, then the theta series \eqref{conj 4.3} has non-negative coefficients. 
 \end{theorem}

 Let\begin{equation}\label{eqn:rho}
 \rho_{R}(n):=\begin{cases}(-1)^i & \text{ if } n=Ri(3i-1)/2 \text{ for some } i\in \mathbb Z \\ 0 & \text{ otherwise.}\end{cases}
 \end{equation}
 Then, \eqref{epnt} with $q$ replaced by $q^R$,   implies that $(q^R;q^R)_\infty$ is the generating function for the sequence $\rho_R(n)$.
 
 When $R=3, S=1$, 
 for all $n\geq 0, k\geq 1$, we have  \begin{equation}\label{conj 4.3-b3} a_{1,3,k}(n)=(-1)^{k-1}\sum_{j=-(k-1)}^k(-1)^j b_3(n-j(3j-1)/2) - (-1)^{k-1}\rho_3(n), \end{equation} where $b_\ell(n)$ denotes the number of $\ell$-regular partitions of $n$, i.e., partitions with no parts divisible by $\ell$.

\begin{remark} From \cite[Lemma 1.2]{AM} it follows that the $q$-series on the right hand side of  \eqref{trunc-equiv} for $R=3, S=1$ equals $$(q^3;q^3)_\infty\sum_{n\geq 0}^\infty M_k(n)q^n.$$  Hence, if Conjecture \ref{Conj 4.3} is true for $R=3$ and $S=1$, then for $k\geq 1$, $$\sum_{j=0}^\infty (-1)^j(M_k(n-T_{3j})-M_k(n-T_{3j+2})) \geq 0,$$ 
where $T_j=j(j+1)/2$ is the $j$th triangular number.
 \end{remark} 
 
  We prove  the following related theorem, which is weaker than Conjecture \ref{Conj 4.3} for $R=3$, $S=1$.

\begin{theorem}\label{thm:b3} Let $k\geq 1$. Then for all $n\geq 0$, $$(-1)^{k}\sum_{j=-k}^k(-1)^j b_3(n-j(3j-1)/2) +(-1)^{k-1} \rho_{3}(n)\geq 0.$$
\end{theorem}

Note that Theorem \ref{thm:b3} together with  \eqref{conj 4.3-b3} shows that for all $n\geq 0$, $k\geq 1$,  $$a_{1,3,k}(n)+b_3(n-k(3k-1)/2)\geq 0.$$ Following the proof of Theorem \ref{thm:b3}  we also give a combinatorial interpretation for $a_{1,3,k}(n)+b_3(n-k(3k-1)/2)$.

Inequalities for $b_6(n)$ analogous to those in Theorem \ref{thm:b3} were conjectured in \cite{BM3}. We give a proof of the conjecture. 
\begin{theorem}{\cite[Conjecture 3]{BM3}}\label{thm:b6}
    For $n\geq 0, k>0$
    \begin{equation}\label{eq:b6-conj}(-1)^{k-1}\left(\rho_{6}(n)-\sum_{j=-k}^k (-1)^jb_6(n-j(3j-1)/2)\right)\geq 0,\end{equation} with strict inequality if $n\geq (k+1)(3(k+1)-1)/2$.
\end{theorem}
We also give a combinatorial interpretation of the numbers on the left hand side of \eqref{eq:b6-conj}.
\begin{remark} In \cite[Conjecture 3]{BM3}, the inequality \eqref{eq:b6-conj} is conjectured to be strict for $n\geq k(3k+1)/2$. This is, in fact, not true as seen from the proof of Theorem \ref{thm:b6}.
\end{remark} 

When $R=4, S=1$, by considering the parity of triangular numbers, we see that Conjecture \ref{Conj 4.3}  is equivalent to Conjecture 4 in \cite{M3}, which states that for $k\geq 1$ the $q$-series $$(-q;q)_\infty\sum_{n=0}^\infty(-1)^{T_n}q^{T_{n+2k}}$$ has non-negative coefficients.  Here we used Euler's identity $$\frac{1}{(q;q^2)_\infty}=(-q;q)_\infty.$$ 

We prove that a weaker version of the conjecture is eventually true. As usual, we denote by $Q(n)$ the number of  partitions of $n$ with distinct parts, which by Euler's identity is also the number of partitions of $n$ with odd parts. 
Let $ c_1(n)=Q(n)$   and, if $k\geq 2$, let $c_k(n)$ be the sequence whose generating function is \begin{equation}\label{ck} \sum_{n=0}^\infty c_k(n)q^n= (-1)^{k-1}(-q;q)_\infty\left((-1)^{k-1}q^{T_{2(k-1)}}+\sum_{j=0}^{k-2} (-1)^j(q^{T_{2j}}-q^{T_{2j+1}})\right).\end{equation}

\begin{theorem} \label{thm:Q} If $k\geq 1$ there exists $N_k$ such that  $c_k(n)\geq 0$ if $n\geq N_k$. Moreover, $N_k$ can be given explicitly. 
    \end{theorem}

This article is structured as follows. In section \ref{notation}, we introduce relevant background and notation used throughout. In section \ref{sec-main}, we prove Theorem \ref{main} and discuss applications in the case $R=4, S=1$. In particular, in Proposition \ref{prop-mex} we give a connection with the odd minimal excludant of  partitions into odd parts. In section \ref{weak-conj}, we prove Theorems  \ref{thm:b3},  \ref{thm:b6}, and \ref{thm:Q}. In section \ref{sec-JTP} we prove combinatorially an identity related to the Jacobi Triple product identity. In particular cases, this identity reduces to identities (6.7), (6.10), (6.12), and (6.13) conjectured in \cite{M1}.  Hence, we have combinatorial proofs of Merca's identities. In section \ref{distinct-reg}, we use the Weierstrass addition formula \cite[Identity (8.1)]{KMR} to prove an identity involving the number of $5$-regular distinct partition of $n$.   Finally, in section \ref{concluding}, we offer some concluding remarks.

 \section{Preliminaries and notation}\label{notation}
 In this section we discuss some background on partitions and set the notation used frequently in the article.  

We write $\lambda \vdash n$ to mean that $\lambda$ is a partition of $n$. We  say that $\lambda$ has size $n$ if the parts of $\lambda$ add up to $n$ and write $|\lambda|$ for the size of $\lambda$. The length of $\lambda$, denoted $\ell(\lambda)$, is the number of parts in $\lambda$. We denote by $\ell_e(\lambda)$ (respectively $\ell_o(\lambda)$) the number of even (respectively odd) parts in $\lambda$. 
We also work with vectors of partitions and write $(\lambda^{(1)}, \lambda^{(2)}, \ldots, \lambda^{(t)})\vdash n$ to mean that $\lambda^{(1)}, \lambda^{(2)}, \ldots, \lambda^{(t)}$ are partitions with $|\lambda^{(1)}|+|\lambda^{(2)}|+\cdots+|\lambda^{(t)}|=n.$

  We denote by $\mathcal P(n)$ the set of all partitions of $n$.  In general, if $\mathcal A(n)$ denotes a set of partitions of $n$ satisfying certain conditions, we set $\mathcal A:=\bigcup_{n\geq 0}\mathcal A(n).$ Moreover, if $f$ is a partition counting function, say $f(n)=|\mathcal A(n)|$, we use the convention that $f(x)=0$ if $x\not\in \mathbb N_0$. We use the notation  $\mathbb N_0$ for the set of non-negative integers. 
  
The {Ferrers diagram} of a partition $\lambda=(\lambda_1, \lambda_2, \ldots, \lambda_\ell)$ is an array of left justified boxes such that the $i$th row from the top contains $\lambda_i$ boxes. For example,  the Ferrers diagram of $\lambda=(5, 3, 2, 2, 1)$  is shown below  \medskip

 \begin{center}\small{\ydiagram[*(white)]
{5,3,2,2,1}}\end{center}
\medskip

We abuse notation and write $\lambda$ to mean a partition, its Ferrers diagram, or  its  multiset of  parts.   We write $a\in \lambda$ to mean the positive integer $a$ is a part of $\lambda$. The number of times $a>0$ appears in $\lambda$ is denoted by $m_\lambda(a)$ and is called the multiplicity of $a$ in $\lambda$. When convenient, we use the frequency notation for partitions: the exponent of a part is the multiplicity of the part. For example, $\lambda=(5^2, 3^4, 1)$ is the partitions $(5,5,3,3,3,3,1)$. 
We also use the convention that $\lambda_j=0$ for all $j>\ell(\lambda)$ and set $\lambda_0:=\infty$.   

We define the following operations on partitions. If   $\lambda$ and $\mu$ are partitions, by $\lambda\cup \mu$ and $\lambda\setminus \mu$ we mean the corresponding operations on the multisets of parts of $\lambda$ and $\mu$. Moreover $\lambda\setminus \mu$ is only defined if $\mu \subseteq \lambda$ as multisets. 

Given an integer $R\geq 2$, and a partition $\lambda$, the $R$-modular diagram of $\lambda$ is a Ferrers diagram with $\ell(\lambda)$ rows such that, if $\lambda_i=q_iR+S_i$ for $q_i\geq 0$ and $0< S_i\leq R$, then the $i$th row from the top has $q_i+1$ boxes: the first $q_i$ boxes are filled with $R$ and the last box is filled with $S_i$. 
For example,  if $R=3$ and $\lambda=(10, 9, 8,8, 7, 3, 2)$, the $3$-modular diagram of $\lambda$ is 
 $${\young(3331,333,332,332,331,3,2)}$$

Given  integers $R\geq 2$ and $0<S<R$, we denote by $\mathcal P_{\pm S,R}(n)$ the set of partitions of $n$ with parts congruent to $\pm S$ modulo  $R$. For $j\geq 1$,  we denote by $\delta_{S,R, j}$ the  partition in $\mathcal P_{\pm S,R}$ with distinct parts equal to $iR+S$ for $0\leq i\leq j-1$. If $j=0$,  we set $\delta_{S,R,0}:=\emptyset$. The $R$-modular Ferrers diagram of $\delta_{S, R, j}$  is a staircase shaped diagram with $j$ rows,  each  having  $S$ in the last box. For example, the $5$-modular diagram of $\delta_{3, 5, 4}$ is $$\young(5553,553,53,3)$$

Moreover, $|\delta_{S, R, j}|=j(j-1)R/2+jS$. If $R$ is clear from the context, we drop the subscript $R$ from $\delta_{S,R, j}$ to write $\delta_{S, j}$. 

 For a partition $\lambda$, we write $\lambda_{\ell-i}$ for the $(i+1)^{\textrm{st}}$ part counted from the last: $\lambda_\ell$ is the last part, $\lambda_{\ell-1}$ is the second to last part, etc. 
 
 If $r\geq 2$ and $0\leq s<r$ we write $\lambda^{s,r}$ for the partitions whose parts are the parts of $\lambda$ that are congruent to $s$ modulo $r$ (with the same multiplicity as in $\lambda$) and we set $\lambda^{\pm s, r}:=\lambda^{s,r}\cup \lambda^{r-s, r}$. 

 We use the customary shorthand notation $$(a_1,a_2, \ldots, a_j;q)_\infty:=(a_1;q)_\infty(a_2;q)_\infty\cdots(a_j;q)_\infty.$$

 \section{Proof of Theorem \ref{main} and applications}\label{sec-main}

\begin{proof}[Proof of Theorem \ref{main}] Let $R\geq 2$ and $0<S<R$. First we consider the case $S\neq R/2$. 
Recall that  Conjecture \ref{Conj 4.3}  states that, for all $k\geq 1$, the sequence $a_k(n)$ defined by \eqref{trunc-equiv} has non-negative terms. We prove the conjecture for $k\in \{1,2,3\}$ using combinatorial arguments.

 Let $a_0(n)=\rho_R(n)$. 
 From \eqref{eqn:rho}, we see that $$a_0(n)=\begin{cases} 1 & n=Ri(3i-1)/2 \text{ for $i$ even}\\
 -1 & n=Ri(3i-1)/2 \text{ for $i$ odd}
 \\
 0 & \text{otherwise.}\end{cases}$$
 
  The  sequence  $Ri(3i-1)/2$ with $i$  even starts: $0, 5R, 7R, 22R, 26R, \ldots.$\\
The sequence  $Ri(3i-1)/2$ with $i$  odd starts: $R, 2R, 12R, 15R, \ldots.$

For $k\geq 0$, \begin{align} \notag\sum_{n=0}^\infty \big(a_k(n)+ a_{k+1}(n)\big)q^n & =\frac{q^{k(k+1)R/2-kS}}{(q^S,q^{R-S};q^R)_\infty} (1-q^{(2k+1)S})\\ \label{sum of terms} & = \frac{q^{k(k+1)R/2-kS}}{(q^{R+S},q^{R-S};q^R)_\infty} (1+q^S+q^{2S}+\cdots +q^{2kS}). \end{align}

The $q$-series \eqref{sum of terms} is the generating function for $y_k(n):=|\mathcal Y_k(n)|$, where  
 $$\mathcal Y_k(n):=\{(\lambda, \delta_{R-S,R, k})\vdash n\mid \lambda\in \mathcal P_{\pm S,R},0\leq m_\lambda(S)\leq 2k\}.$$ Hence, for $k\geq 0$, we have $a_{k+1}(n)=y_k(n)-a_k(n)$. \smallskip

\noindent \textbf{Case $k=1$:} Let $n\geq 0$. Then $a_1(n)= y_0(n)-a_0(n)$.
 We see that $\emptyset \in \mathcal Y_0(0)$ and if $n=mR$ with $m\geq 5$, then $\{(m-1)R+S, R-S\}\in \mathcal Y_0(n)$. 
Hence, $y_0(n)\geq a_0(n)$ for all $n\geq 0$ and thus $a_1(n)\geq 0$ for all $n\geq 0$.
\medskip

\noindent \textbf{Case $k=2$:} We have $a_2(n)=y_1(n)-a_1(n)=y_1(n)-y_0(n)+a_0(n).$ If $n=0$, then $\mathcal Y_0(0)=\{(\emptyset, \emptyset)\}$, $\mathcal Y_1(0)=\emptyset$ and $a_0(0)=1$. Thus $a_2(0)=0$.

If $n\geq 1$ we create an injection 
$$\varphi_1: \mathcal Y_0(n) \to \mathcal Y_1(n)=\{(\lambda, \delta_{R-S,1})\vdash n\mid \lambda\in \mathcal P_{\pm S,R}\},   0\leq m_\lambda(S)\leq 2\}.$$ Let $(\lambda,\emptyset)\in \mathcal{Y}_0(n)$. Thus, $m_\lambda(S)=0$. 
\smallskip

\underline{Case I:}  $m_\lambda(R-S)\geq 1$. Define $\varphi_1(\lambda, \emptyset):=(\lambda\setminus (R-S), \delta_{R-S, 1})$.

\smallskip

 \underline{Case II:} $m_\lambda(R-S)=0$. Since $n\geq 1$, we must have $\lambda\neq \emptyset$ and  $$\lambda_\ell\geq \min\{R+S, 2R-S\}.$$
Define $\varphi_1(\lambda, \emptyset):=(\mu, \delta_{R-S, 1})$, where $\mu$ is the partition obtained from $\lambda$ by replacing part $\lambda_\ell$  by parts $\lambda_\ell-R, S$.

In each of Cases I and II, the transformation $\varphi_1$ is invertible on its image. 
Considering $m_\mu(S)$ for $(\mu, \delta_{R-S, 1})\in \varphi_1(\mathcal Y_0(n))$, we see that  the image of $\varphi_1$ in Case I is disjoint from the image of $\varphi_1$ in Case II. Thus, for $n\geq 1$, $y_1(n)-y_0(n)\geq 0$. 

Finally, to see that $a_2(n)\geq 0$ for all $n\geq 0$, we need to show that for each $n=Ri(3i-1)/2$ with $i$ odd,  $\mathcal B_1(n):=\mathcal Y_1(n)\setminus \varphi_1\big(\mathcal Y_0(n)\big)\neq \emptyset$. Identifying partitions with their multiset of parts, we have   $((S), \delta_{R-S,1})\in \mathcal B_1(R)$ and   $(\{R-S,S^2\},\delta_{R-S,1}) \in \mathcal B_1(2R)$. If $n=mR$ with $m\geq 12$, then  $(\{(m-3)R+S, (R-S)^2, S^2\},\delta_{R-S,1})\in \mathcal B_1(mR)$.
Hence, $a_2(n)\geq 0$ for all $n\geq 0$.\smallskip

\medskip

\noindent \textbf{Case $k=3$:} We have $a_3(n)=y_2(n)-a_2(n)=y_2(n)-(y_1(n)-y_0(n))-a_0(n).$ Suppose $S<R/2$. 
First observe that 
\begin{itemize}\item if $n=0$, then $\mathcal Y_0(0)=\{(\emptyset, \emptyset)\}$, $\mathcal Y_1(0)=\mathcal Y_2(0)=\emptyset$ and $a_0(0)=1$. Thus, $a_3(0)=0$. 
\item if $n=R$, then $\mathcal Y_0(R)=\mathcal Y_2(R)=\emptyset$, $\mathcal Y_1(R)=\{((S), \delta_{R-S, 1})\}$ and  $a_0(R)=-1$. Thus, $a_3(R)=0$.   
\item  If $n=2R$, then  $\mathcal Y_0(2R)=\{((R+S, R-S), \emptyset)\}$, $\mathcal Y_1(2R)=\{((R+S), \delta_{R-S, 1}), ((R-S, S^2), \delta_{R-S, 1})\}$, $\mathcal Y_2(2R)=\emptyset$ and     $a_0(2R)=-1$. Thus,  $a_3(2R)=0$. \end{itemize}

Let $n\not \in \{0,R, 2R\}$. 
We create an injection $$\varphi_2:\mathcal B_1(n)\to\mathcal Y_2(n)=\{(\lambda, \delta_{R-S,2})\vdash n \mid \lambda\in \mathcal P_{\pm S,R},   0\leq m_\lambda(S)\leq 4\}.$$ We consider several cases and in each case we describe the image of $\varphi_2$ either precisely or as a subset of a larger set. This will be sufficient to conclude that $\varphi_2$ is an injection.

Let $(\lambda, \delta_{R-S,1})\in \mathcal B_1(n)$. From Case $k=2$,   $\mathcal B_1(n)$  consists of the  pairs $(\mu, \delta_{R-S,1})\in \mathcal Y_1(n)$ with $\mu$ satisfying $\ell(\mu)\geq 3$,  $m_\mu(S)\in \{1,2\}$,  $\mu_{\ell-2}-\mu_{\ell-1}< R$.

\smallskip

\underline{Case (i)} $m_\lambda(R-S)\geq 2$. Thus,  $\ell(\lambda)\geq 3$ and $\lambda_{\ell-2}-\lambda_{\ell-1}<R$. 
We define $\varphi_2(\lambda, \delta_{R-S,1}):=(\lambda\setminus(R-S,R-S, S), \delta_{R-S,2})$.

 We obtain all pairs $(\mu, \delta_{R-S,2})\in \mathcal Y_2(n)$ with 
 $m_\mu(S)\in \{0,1\}$ 

\smallskip

\underline{Case (ii)} $m_\lambda(R-S)= 0$. The condition $\lambda_{\ell-2}-\lambda_{\ell-1}< R$ implies that $m_\lambda(S)=1$.  Moreover,  $\lambda_{\ell-1}\geq R+S$. Define $\varphi_2(\lambda, \delta_{R-S,1}):=(\mu, \delta_{R-S, 2})$, where $\mu$ is the partition obtained from $\lambda$ by replacing parts $\lambda_{\ell-1}$ and $\lambda_{\ell-2}$ by parts $\lambda_{\ell-1}-R$ and $\lambda_{\ell-2}-R, S$.

We obtain all pairs $(\mu, \delta_{R-S, 2})\in \mathcal Y_2(n)$ satisfying all of the following conditions. 
\begin{itemize}\item $\ell(\mu)\geq 4$; \item  $\mu_{\ell-4}-\mu_{\ell-3}\geq R$; \item $\mu_{\ell-3}-\mu_{\ell-2}< R$;
\item $2\leq m_\mu(S)\leq 4$. 
\end{itemize}
\smallskip

\underline{Case (iii)}  $m_\lambda(R-S)= 1$. 
\smallskip

(a)  $m_\lambda(S)=2$.  Since $n\neq 2R$,  we must have $\ell(\lambda)\geq 4.$
Define $\varphi_2(\lambda, \delta_{R-S,1}):=(\mu, \delta_{R-S, 2})$, where $\mu$ is the partition obtained from $\lambda$ by  replacing parts $\lambda_{\ell-3}$ and $R-S$ by part $\lambda_{\ell-3}-R$.

We obtain  pairs $(\mu, \delta_{R-S, 2})\in \mathcal Y_2(n)$ satisfying all of the following conditions. 
\begin{itemize}\item $\ell(\mu)\geq 3$; \item  $\mu_{\ell-3}-\mu_{\ell-2}\geq R$; 
\item $2\leq m_\mu(S)\leq 3$. 
\end{itemize}
\smallskip

(b) $m_\lambda(S)=1$. Then, condition $\lambda_{\ell-2}-\lambda_{\ell-1}<R$ implies that $\lambda_{\ell-2}=R+S$. First, let $\xi$ be the partition obtained from $\lambda$ by replacing parts $R+S$ and $R-S$ by a part equal to $S$. We have $m_\xi(S)=2$ and $\xi_{\ell-2}\geq R+S$. 

If $\ell(\xi)=2$ or $\ell(\xi)\geq 4$ and $\xi_{\ell-i}-\xi_{\ell-i+1}<R$ for $i\in \{3,4\}$, define $\varphi_2(\lambda, \delta_{R-S, 1}):=(\xi,\delta_{R-S, 2})$. 

We obtain  pairs $(\mu, \delta_{R-S,2})\in\mathcal Y_2(n)$  satisfying all of the following conditions.
\begin{itemize}
   \item $\ell(\mu)\neq 3$;
   \item if $\ell(\mu)\geq 4$, then $\mu_{\ell-4}-\mu_{\ell-3}< R$  and $\mu_{\ell-3}-\mu_{\ell-2}< R$;
    \item $m_\mu(S)=2$.
   \end{itemize}

Otherwise, if $\xi_{\ell-2}\leq 2R+S$, define $\varphi_2(\lambda, \delta_{R-S,1}):=(\mu, \delta_{R-S, 2})$, where $\mu$ is the partition obtained from $\xi$ by replacing part $\xi_{\ell-2}$ by parts $\xi_{\ell-2}-R, R-S, S$.

We obtain pairs $(\mu, \delta_{R-S,2})$ in $\mathcal Y_2(n)$ with 
\begin{itemize}
\item $\ell(\mu)\geq 5$;
\item $(m_\mu(S), m_\mu(R-S))\in \{(3,1), (3,2), (4,1)\}$;
\item $\mu_{\ell-4}-\mu_{\ell-3}<R$; \item  $\mu_{\ell-5}-\mu_{\ell-4}\geq R$.
\end{itemize}

If  $\xi_{\ell-2}\geq 2R+R-S$, Define $\varphi_2(\lambda, \delta_{R-S,1}):=(\mu, \delta_{R-S, 2})$, where $\mu$ is the partition obtained from $\xi$ by replacing part $\xi_{\ell-2}$ by parts $\xi_{\ell-2}-2R, R-S,R-S, S,S$. 

We obtain pairs $(\mu, \delta_{R-S,2})$ in $\mathcal Y_2(n)$ with 
\begin{itemize}
\item $\ell(\mu)\geq 7$;
\item $m_\mu(S)=4$ and $2 \leq m_\mu(R-S)\leq 3$;
\item $\mu_{\ell-7}-\mu_{\ell-6}\geq 2R$.
\end{itemize}
The second condition implies $\mu_{\ell-4}-\mu_{\ell-3}<R$ and $\mu_{\ell-5}-\mu_{\ell-4}<R$.
\smallskip

In each case, the transformation $\varphi_2$ is invertible on its image and the images in the different cases are disjoint. 
Thus, for $n\not \in \{0, R, 2R\}$, $y_2(n)-(y_1(n)-y_0(n))\geq 0$. 

Let $n\not \in \{0, R, 2R\}$ and let $\mathcal B_2(n)=\mathcal Y_2(n)\setminus \varphi_2(\mathcal B_1(n))$. We need to show that for $n=Ri(3i-1)/2$ with $i$ even, $\mathcal B_2(n)\neq \emptyset$. 

We have $(((R-S)^2, S^4), \delta_{R-S,2})\in \mathcal B_2(5R)$, $(R+S, (R-S)^3, S^4), \delta_{R-S,2})\in \mathcal B_2(7R)$, and for $m\geq 22$, $((m-8)R-S, (R+S)^2, (R-S)^3, S^4), \delta_{R-S,2})\in \mathcal B_2(mR)$.

The discussion above shows that $\varphi_2$ is an injection and thus $a_2(n)\leq y_2(n)$. Hence,  $a_3(n)=y_2(n)-a_2(n)\geq 0$ for all $n\geq 0$. 

Next assume $S>R/2$. Much of this case is similar to the case $S<R/2$. We highlight the differences below.

First note that if $R-S$ divides $R$ (which is equivalent to $R-S$ divides $S$), then in addition to the  pairs in $\mathcal Y_i(n)$ for $0\leq i\leq 2$, $n=0,R, 2R$, listed in the case $S<R/2$, we have 
\begin{align*}
((R-S)^{R/(R-S)}, \emptyset)&\in \mathcal Y_0(R),
\\
((R-S)^{S/(R-S)}, \delta_{R-S,1})&\in \mathcal Y_1(R),\\
((R-S)^{2R/(R-S)}, \emptyset), ((2R-S,(R-S)^{S/(R-S)}),\emptyset)&\in \mathcal Y_0(2R),\\ 
((R-S)^{(R+S)/(R-S)} \delta_{R-S,1}), ((S, (R-S)^{R/(R-S)}),\delta_{R-S,1})  &\in \mathcal Y_1(2R),
\\
((R-S)^{(2S-R)/(R-S)}, \delta_{R-S,2})&\in \mathcal Y_2(2R).
\end{align*}
  Thus, we still have $a_3(n)\geq 0$ for $n\in \{0, R, 2R\}$. 

Let $n\not \in \{0, R, 2R\}$. The set $\mathcal B_1(n)$ consists of the  pairs $(\mu, \delta_{R-S,1})\in \mathcal Y_1(n)$ with $\mu$ satisfying  $\ell(\mu)\geq 3$, $m_\mu(S)\in \{1,2\}$, $(m_\mu(S), m_\mu(R-S))\neq (1,1)$, and  $\mu_{\ell-2}-\mu_{\ell-1}< R$. 

Let $(\lambda, \delta_{R-S,1})\in \mathcal B_1(n)$.

\underline{Case (I)} $m_\lambda(R-S)\geq 2$. Proceed as in Case (i) above and obtain the same image (pairs $(\mu, \delta_{R-S, 2})\in \mathcal Y_2(n)$  with $m_\mu(S)\in \{0,1\}$).\smallskip

\underline{Case (II)} $m_\lambda(R-S)=0$ and $m_\lambda(S)=1$. Proceed as in Case (ii) above. We obtain pairs $(\mu, \delta_{R-S, 2})\in \mathcal Y_2(n)$ satisfying all of the following conditions (which are slightly different than in Case (ii) above).
\begin{itemize}\item $\ell(\mu)\geq 4$;\item $2\leq m_\mu(S)\leq 4$ and  $(m_\mu(S), m_\mu(R-S))\neq(2,1)$;   \item  if $m_\mu(S)=m_\mu(R-S)=2$, then $\mu_{\ell-4}-\mu_{\ell-3}\geq 2R-2S$,\\ else $\mu_{\ell-4}-\mu_{\ell-3}\geq R$; \item $\mu_{\ell-3}-\mu_{\ell-2}< R$.
\end{itemize}
\smallskip

\underline{Case (III)} $m_\lambda(R-S)=1$ and $m_\lambda(S)=2$. Proceed as in Case (iii)(a) above.  We obtain pairs $(\mu, \delta_{R-S, 2})\in \mathcal Y_2(n)$ satisfying all of the following conditions (which are slightly different than in Case (iii)(a) above).
\begin{itemize}\item $\ell(\mu)\geq 3$;\item $2\leq m_\mu(S)\leq 3$;  \item  if $m_\mu(R-S)=1$, then $\mu_{\ell-3}-\mu_{\ell-2}\geq 2R-2S$, \\ else $\mu_{\ell-3}-\mu_{\ell-2}\geq R$. 

\end{itemize}
\smallskip

\underline{Case (IV)} $m_\lambda(R-S)=0$ and $m_\lambda(S)=2$. Then, condition $\lambda_{\ell-2}-\lambda_{\ell-1}<R$ implies that $\lambda_{\ell-2}=2R-S$. First, let $\zeta$ be the partition obtained from $\lambda$ by removing one  part equal to $2R-S$.  We have $m_\zeta(S)=2$ and $\zeta_{\ell-2}\geq 2R-S$. In particular $m_\zeta(R-S)=0$. Now we proceed as in Case (iii)(b) with $\xi$ replaced by $\zeta$. There are again slight modifications for the conditions on the obtained pairs. They are as follows. If $\zeta_{\ell-2}\leq 2R+S$ and in the obtained pair $(\mu, \delta_{R-S, 2})$ we have $(m_\mu(S), m_\mu(R-S))\in\{(3,1), (3,2)\}$, then $\mu_{\ell-5}-\mu_{\ell-4}\geq 2R-2S$, else $\mu_{\ell-5}-\mu_{\ell-4}\geq R$.  If $\zeta_{\ell-2}> 2R+S$, then $\zeta_{\ell-2}\geq 4R-S$ and in the obtained pair $(\mu, \delta_{R-S, 2})$ we have $m_\mu(R-S)=2$.

The rest of the argument
is the same as in the case $S<R/2$. 

\bigskip

If $R$ is even and $S=R/2$  then $1/(q^S,q^{R-S};q^R)_\infty=1/(q^{R/2};q^R)^2_\infty$.
The proof follows in a similar manner as in the case $S<R/2$ above if  instead of  partitions $\lambda\in \mathcal P_{\pm S,R}$ we consider partitions $\lambda\in \mathcal P_{R/2,R}$ in two colors.

Alternatively, if $R$ is even and $S=R/2$, then \eqref{trunc-equiv} becomes
\begin{align*} \sum_{n=0}^\infty a_k(n)q^n:=\frac{(-1)^{k-1}}{(q^S;q^{2S})^2_\infty} 
\sum_{j=0}^{k-1} (-1)^jq^{j^2S}(1-q^{(2j+1)S}) -(-1)^{k-1}(q^{2S};q^{2S})_\infty. 
\end{align*} Replacing $q$ by $q^{1/S}$, we see that proving the non-negativity of the coefficients in the above $q$-series reduces to the case $R=2$, $S=1$ of Conjecture \ref{Conj 4.3}, i.e., 
the $q$-series $$\frac{(-1)^{k-1}}{(q;q^{2})_\infty} \left(\frac{1}{(q;q^{2})_\infty}
\sum_{j=0}^{k-1} (-1)^jq^{j^2}(1-q^{2j+1}) -(q;q)_\infty\right)$$ has non-negative coefficients. When $k=1, 2,3$, one can define  simpler injections to prove the stronger Conjecture \ref{Conj_2S} below. We leave this as an exercise for the reader.

This concludes the proof of Theorem \ref{main}. \end{proof}
\begin{conjecture}\label{Conj_2S}
     Let $k$ be a positive integer. Then, the series $$\frac{(-1)^{k-1}}{(q;q^{2})_\infty}
\sum_{j=0}^{k-1} (-1)^jq^{j^2}(1-q^{2j+1}) -(-1)^{k-1}(q;q)_\infty$$ has non-negative coefficients.
\end{conjecture}

Next, we derive some applications of Theorem \ref{main}. We note that, when  $(R,S)=(3,1)$,  $$\frac{1}{(q^S,q^{R-S};q^R)_\infty}=\frac{1}{(q,q^2;q^3)_\infty} =\sum_{n=0}^\infty b_3(n)q^n$$  and when $(R,S)=(4,1)$,$$\frac{1}{(q^S,q^{R-S};q^R)_\infty} =\frac{1}{(q,q^2;q^4)_\infty} =\sum_{n=0}^\infty Q(n)q^n.$$  
  Then, Theorem \ref{main}  leads to   the following linear inequalities for $b_3(n)$   and also for $Q(n)$.

\begin{corollary} \label{Cor-bQ} For all $n\geq 0$, 
\begin{align*} b_3(n)-b_3(n-1)-\rho_3(n)& \geq 0\\ b_3(n)-b_3(n-1)-b_3(n-2)+b_3(n-5)-\rho_3(n)& \leq 0\\ b_3(n)-b_3(n-1)-b_3(n-2)+b_3(n-5)+ b_3(n-7)-b_3(n-12)-\rho_3(n)& \geq 0\\ Q(n)-Q(n-1)-\rho_4(n)& \geq 0\\ Q(n)-Q(n-1)-Q(n-3)+Q(n-6)-\rho_4(n)& \leq 0\\ Q(n)-Q(n-1)-Q(n-3)+Q(n-6)+ Q(n-10)-Q(n-15)-\rho_4(n)& \geq 0.
\end{align*}
\end{corollary}

We continue with another application of Theorem \ref{main} when $(R,S)=(4,1)$. 

Given a partition $\lambda$, the minimal excludant of $\lambda$, denoted by $\mex(\lambda)$, is the smallest positive integer that does not occur as a part of $\lambda$. For example, $$\mex(7,7,5,3,2,2,1)=4.$$ 

In \cite{AN},  Andrews and Newman define 
 $\mex_{A,a}(\pi)$
to be the smallest integer congruent to $a$ modulo $A$ that is not a part of the partition $\pi$. They also define $p_{A,a}(n)$ to be the number of partitions of $\pi$ of $n$ with $$\mex_{A,a}(\pi)\equiv a\pmod{2A}$$ and $\overline p_{A,a}(n)$ to be the number of partitions  $\pi$ of $n$ with $$\mex_{A,a}(\pi)\equiv A+a\pmod{2A}.$$

As stated in the introduction, when $R=4, S=1$, Conjecture \ref{Conj 4.3}  is equivalent to the non-negativity of the coefficients of  the $q$-series $$\frac{1}{(q;q^2)_\infty}\sum_{n=0}^\infty(-1)^{T_n}q^{T_{n+2k}}.$$ From case $k=1$ of Theorem \ref{main}, we have that for all $n\geq 0$,  \begin{align}\label{Q-mex1} \sum_{j\geq 0}\big(Q(n-T_{4j+2})-Q(n-T_{4j+4})\big)+ \sum_{j\geq 0}\big(Q(n-T_{4j+1}) -& Q(n-T_{4j+3})\big)\\ & \notag \geq Q(n-1).\end{align}

Since $T_{2j}=|\delta_{3,4,j}|$  the transformation $\lambda \mapsto \lambda \cup \delta_{3,4,j}$  shows that \begin{align*}
    Q(n-T_{2j})& =|\{\lambda\vdash n \mid \lambda \text{ has odd parts, } \mex_{4,3}(\lambda)\geq 4j+3\}|
\end{align*}
Similarly, the transformation $\lambda \mapsto \lambda \cup \delta_{1,4,j+1}$ shows that \begin{align*}
    Q(n-T_{2j+1})& =|\{\lambda\vdash n \mid \lambda \text{ has odd parts, } \mex_{4,1}(\lambda)\geq 4j+5\}|.
\end{align*} Hence, $Q(n-T_{4j+2})-Q(n-T_{4j+4})$ is the number of  partitions $\pi$ of $n$ into odd parts with $\mex_{4,3}(\pi)=8j+7$  and $Q(n-T_{4j+1})-Q(n-T_{4j+3})$  is the number of  partitions $\pi$ of $n$ into odd parts with $\mex_{4,1}(\pi)=8j+5$.   

In analogy to \cite{AN}, we denote by  $\overline Q_{A,a}(n)$  the number of partitions  $\pi$ of $n$ into odd parts with $$\mex_{A,a}(\pi)\equiv A+a\pmod{2A}.$$ Then, the discussion above leads to the following result. 

\begin{proposition} \label{prop-mex} For $n\geq 0$, we have $$\overline Q_{4,3}(n)+ \overline Q_{4,1}(n)\geq Q(n-1).
$$
\end{proposition}

\section{Proof of Theorems  \ref{thm:b3},  \ref{thm:b6}, and \ref{thm:Q}}\label{weak-conj}

\subsection{Proof of Theorem \ref{thm:b3}} We prove that for fixed $k\geq 1$ and all $n\geq 0$, 
$$(-1)^{k}\sum_{j=-k}^k(-1)^j b_3(n-j(3j-1)/2) +(-1)^{k-1} \rho_3(n)\geq 0.$$

From \cite[Corollary 17]{M4}, we have 
\begin{align}\label{Shanks}
    \frac{(-1)^k}{(q;q)_\infty}& \sum_{j=-k}^k(-1)^jq^{j(3j-1)/2}  -(-1)^k\\ \notag = &  \frac{q^{k(3k+7)/2+2}}{(q;q^3)_\infty(q^3;q^3)_\infty}\sum_{n=0}^\infty \frac{q^{n(3n+3k+5)}}{(q^3;q^3)_n(q^2;q^3)_{n+k+1}}\\\notag  & + \frac{q^{k(3k+5)/2+1}}{(q^2;q^3)_\infty(q^3;q^3)_\infty}\sum_{n=0}^\infty \frac{q^{n(3n+3k+4)}}{(q^3;q^3)_n(q;q^3)_{n+k+1}}.
\end{align}
Multiplying both sides of \eqref{Shanks} by $(q^3;q^3)_\infty$ we obtain 
\begin{align}\label{b3:qseries}
   (-1)^k \frac{(q^3;q^3)_\infty}{(q;q)_\infty}& \sum_{j=-k}^k(-1)^jq^{j(3j-1)/2}  +(-1)^{k-1}(q^3;q^3)_\infty\\ = & \notag \frac{q^{k(3k+5)/2+1}}{(q^2;q^3)_\infty}\sum_{n=0}^\infty \frac{q^{n(3n+3k+4)}}{(q^3;q^3)_n(q;q^3)_{n+k+1}}
   \\ &\notag + \frac{q^{k(3k+7)/2+2}}{(q;q^3)_\infty}\sum_{n=0}^\infty \frac{q^{n(3n+3k+5)}}{(q^3;q^3)_n(q^2;q^3)_{n+k+1}}\\ & =: H_1(q)+ H_2(q).
\end{align}

Since each of $H_1(q)$ and $H_2(q)$ above   has non-negative coefficients, this completes the proof Theorem \ref{thm:b3}. \qed \medskip

We now interpret \eqref{b3:qseries} as a partition generating function.  First we introduce some notation. Given a $3$-regular partition $\mu$, we define the $(k+2,3)-$Durfee rectangle of $\mu$ to be  the largest $a\times (a+k+2)$ rectangle that fits inside the $3$-modular diagram of $\mu$. Note that $\sum_{t=0}^k (3t+2)=k(3k+7)/2+2$ and $\sum_{t=0}^k (3t+1)=k(3k+5)/2+1$.

Let $\mathcal B_3^i(m)$ be 
the set of $3$-regular partitions $\lambda$ of $m$ such that if we write $\lambda=\lambda^{1,3}\cup\lambda^{2,3}$,  parts $3t+i$,  $0\leq t\leq k$ occur in $\lambda^{i,3}$  at least once and the   first row in the $3$-regular diagram of $\lambda^{i,3}$  below the $(k+2,3)-$Durfee rectangle is shorter than the width of the $(k+2,3)-$Durfee rectangle. (If the $(k+2,3)-$Durfee rectangle of $\lambda^{i,3}$  has height zero, this means that the parts of $\lambda^{i,3}$  are at most  $3k+i$ and all parts $3t+i$,  $0\leq t\leq k$ occur.) Then $H_i(q)$,   $i=1,2$,  is the generating function for the sequence $|\mathcal B_3^i(n)|$. To see the interpretation of $H_1(q)$  as a partition generating function,  write $$H_1(q)=\frac{1}{(q^2;q^3)_\infty}\sum_{n=0}^\infty \frac{q^{k(3k+5)/2+1}}{(q;q^3)_{n+k+1}} \cdot \frac{q^{n(3n+3k+4)}}{(q^3;q^3)_n}=: \frac{1}{(q^2;q^3)_\infty}\sum_{n=0}^\infty h_1(q;n)\cdot h_2(q;n).$$ Informally, for partitions $\lambda \in \mathcal B_3^1(n)$,
\begin{itemize} \item the term $1/{(q^2;q^3)_\infty}$ generates the parts of $\lambda^{2,3}$; \item  $h_2(q;n)$ generates the first $n$ parts of $\lambda^{1,3}$, all of size at least $3n+3k+4$; 
\item $h_1(q;n)$ generates the remaining  parts of $\lambda^{1,3}$, all of size at most $3n+3k+1$ and parts $3t+1$, $0\leq t\leq k$, occur at least once.  
\end{itemize}
Thus, term $h_1(q;n)\cdot h_2(q;n)$  generates partitions $\lambda^{1,3}$ with $(k+2, 3)-$Durfee rectangle of height $n$. 
The interpretation of $H_2(q)$ is explained similarly.

Then, the coefficient of $q^n$ is $|\mathcal B_3^1(n)|+|\mathcal B_3^2(n)|$.

We note that $\mathcal B_3^1(n)\cap \mathcal B_3^2(n)\neq \emptyset$.

\subsection{Proof of Theorem \ref{thm:b6}} The generating function for the sequence given by the left hand side of inequality \eqref{eq:b6-conj} is given by $$H(q)=(-1)^{k-1}(q^6;q^6)_\infty+ (-1)^k\frac{(q^6;q^6)_\infty}{(q;q)_\infty}\sum_{j=-k}^k (-1)^jq^{j(3j-1)/2}. $$
Multiplying both sides of \eqref{Shanks} by $(q^6;q^6)_\infty$, we obtain 
\begin{align*}
   (-1)^{k-1}(q^6;q^6)_\infty +& (-1)^k \frac{(q^6;q^6)_\infty}{(q;q)_\infty} \sum_{j=-k}^k(-1)^jq^{j(3j-1)/2}  \\ = &  \frac{q^{k(3k+7)/2+2}(-q^3;q^3)_\infty}{(q;q^3)_\infty}\sum_{n=0}^\infty \frac{q^{n(3n+3k+5)}}{(q^3;q^3)_n(q^2;q^3)_{n+k+1}}\\ & + \frac{q^{k(3k+5)/2+1}(-q^3;q^3)_\infty}{(q^2;q^3)_\infty}\sum_{n=0}^\infty \frac{q^{n(3n+3k+4)}}{(q^3;q^3)_n(q;q^3)_{n+k+1}}, 
\end{align*} which has non-negative coefficients.  

Moreover, the coefficient of $q^n$ in the series above equals $|\mathcal B_3^{2*}(n)|+|\mathcal B_3^{1*}(n)|$, where,  for $i=1,2$, $\mathcal B_3^{i*}(n)$ is 
the number of 
partitions $\lambda$ of $n$ such that $\lambda^{0,3}$ has distinct parts and $\lambda^{1,3}\cup \lambda^{2,3}\in \mathcal B_3^i$.

\subsection{Proof of Theorem \ref{thm:Q}}

It follows from \eqref{ck} that proving the non-negativity of $c_k(n)$  is equivalent to proving that, for all $n\geq 0$,
if $k$ is odd, $$Q(n-T_{2k-2})+\sum_{i=0}^{\frac{k-3}{2}}\big(Q(n-T_{4i})+Q(n-T_{4i+3})\big)\geq \sum_{i=0}^{\frac{k-3}{2}}\big(Q(n-T_{4i+1})+Q(n-T_{4i+2})\big)$$ and if $k$ is even 
$$Q(n-T_{2k-1})+\sum_{i=0}^{\frac{k-2}{2}}\big(Q(n-T_{4i+1})+Q(n-T_{4i+2})\big)\geq \sum_{i=0}^{\frac{k-2}{2}}\big(Q(n-T_{4i})+Q(n-T_{4i+3})\big).$$
For each fixed $k$,  \cite[Theorem 1.11]{BBCFW} shows that the inequality is eventually true and provides an explicit bound $N_k$  such that the inequality is true for $n\geq N_k$.

  \section{Combinatorial proof of an identity related to the Jacobi Triple Product}\label{sec-JTP}

  In this section we give both  analytic and  combinatorial proofs of an identity which includes as particular cases 
   the conjectural identities (6.7), (6.10), (6.12), (6.13) in  \cite{M1}.

First, we recall the Jacobi Triple Product identity in its general form. See, for example, \cite{KK} where a combinatorial proof is provided. 
    For integers $m$ and $s$ such that $0<s<m$, and complex numbers $q$ and $z$ with $|q|<1$ and $z\neq 0$, we have 
  \begin{equation}\label{JTP} (zq^s,z^{-1}q^{m-s}, q^m;q^m)_\infty=\sum_{i=-\infty}^\infty(-z)^iq^{m(i^2-i)/2+si}.
      \end{equation}

  \begin{theorem}\label{T-gen}Let $m, s$ be positive integers such that $m\geq 2$ and $0<s<m$. Then, \begin{equation}\label{I6.7-gen}\frac{(q^s,q^{2m-s},q^{2m};q^{2m})_\infty(q^{2m-2s},q^{2m+2s} ;q^{4m})_\infty}{(q^s, q^{m-s};q^{m})_\infty}=\sum_{n=-\infty}^\infty q^{n(mn+s)}.\end{equation}
  \end{theorem}

  \begin{corollary}If $m=5$, then for $s=1,2,3,4$, we obtain the respective identities (6.7), (6.10), (6.12), (6.13) in  \cite{M1}.   \end{corollary}
  \begin{remark} In \cite{KMR}, the authors 
  point out  that identities (6.7), (6.10), (6.12), (6.13) in  \cite{M1} follow from the Jacobi Triple Product identity. 
  \end{remark}

  \begin{proof}[Analytic proof]  Since   \begin{align*}(q^s;q^{m})_\infty& =(q^s;q^{2m})_\infty(q^{m+s};q^{2m})_\infty\\ (q^{2m+2s};q^{4m})_\infty& =  (q^{m+s};q^{2m})_\infty(-q^{m+s};q^{2m})_\infty,\end{align*} to prove identity \eqref{I6.7-gen}, we need to show that \begin{align}\label{eq_staircase}(-q^{m+s};q^{2m})_\infty(-q^{m-s};q^{2m})_\infty(q^{2m};q^{2m})_\infty=\sum_{n=-\infty}^\infty q^{n(mn+s)},
  \end{align} which is precisely \eqref{JTP} with $q$ replaced by $q^2$ and $z$ replaced by $-q^{m-s}$.
      \end{proof}
 \begin{proof}[Combinatorial proof] 
 To simplify the exposition, we first   define a sign reversing involution to prove combinatorially the trivial cancellation $$\frac{(q;q)_\infty}{(q;q)_\infty}=1.$$ We say that an involution defined on a set of pairs (or tuples) of partitions is sign reversing, if the involution changes  the parity of the length of the first partition in each pair (or tuple). Let $\mathcal{QP}(n)$ be the set of pairs of partitions $(\lambda, \mu)\vdash n$ with $\lambda$ a distinct partition and $\mu$ an unrestricted partition. The map $F:\mathcal{QP}(n)\setminus \{(\emptyset, \emptyset)\} \to \mathcal{QP}(n)\setminus \{(\emptyset, \emptyset)\}$ defined by
 $$F(\lambda, \mu)=\begin{cases}(\lambda\cup(\mu_1), \mu\setminus(\mu_1)) & \text{ if } \lambda_1<\mu_1 \\ (\lambda\setminus(\lambda_1), \mu\cup(\lambda_1)) & \text{ if } \lambda_1 \geq \mu_1 \end{cases} $$ is a sign reversing involution. Recall that, following our  convention,    if $\nu=\emptyset$, then $\nu_1=0$. Since $(q;q)_\infty/(q;q)_\infty$ is the generating function for $$qp_{e-o}(n):=|\{(\lambda, \mu)\in \mathcal{QP}(n), \ell(\lambda) \textrm{ even}\}|-|\{(\lambda, \mu)\in \mathcal{QP}(n), \ell(\lambda) \textrm{ odd}\}|,$$ this completes the argument.

  We continue with the combinatorial proof of    Theorem \ref{T-gen}.   

    If $m=2$, then \eqref{I6.7-gen} becomes \begin{equation}\label{I6.7-2}\frac{(q,q^{3},q^{4};q^{4})_\infty(q^{2},q^{6} ;q^{8})_\infty}{(q;q^{2})^2_\infty}=\sum_{n=-\infty}^\infty q^{n(2n+1)}.\end{equation} Using the transformation $F$ above on pairs of partitions with odd parts, we have that \eqref{I6.7-2} is equivalent to \begin{equation}\label{I6.7-2'}\frac{(q^{2};q^{2})_\infty}{(q;q^{2})_\infty}=\sum_{n=0}^\infty q^{n(n+1)/2}.\end{equation} This is Gauss's identity \cite[Identity (1.3)]{A72} with $q$ replaced by $-q$. However, since $\ell_o(\lambda)\equiv |\lambda| \mod 2$, Andrews' combinatorial proof  of Gauss' identity given in \cite{A72} also proves \eqref{I6.7-2} combinatorially. 
    
    For the remainder of the proof let $m\geq 3$.

   Denote by $\mathcal A(n)$ the set consisting of pairs of partitions $(\lambda, \mu)\vdash n$ such that 
  \begin{itemize}\item $\lambda$ is a partition with distinct parts  congruent to $0$, $s$ or  $-s$ modulo $2m$ or congruent to $2m+2s$ or $2m-2s$ modulo $4m$;

  \item $\mu$ is a partition with parts congruent to $s$ or $-s$ modulo $m$. (Thus, $\mu$ has parts congruent to $s,m+s,-s$ or $m-s$ modulo $2m$.)
  \end{itemize}
Define $$\mathcal A_e(n):=\{(\lambda, \mu)\in \mathcal A(n) \mid \ell(\lambda) \text{ even}\}$$ and $$\mathcal A_o(n):=\{(\lambda, \mu)\in \mathcal A(n) \mid \ell(\lambda) \text{ odd}\}.$$
  
  The left hand side of \eqref{I6.7-gen} is the generating function for the sequence $a_{e-o}(n):=|\mathcal A_e(n)|-|\mathcal A_o(n)|$.

We define a sign reversing  involution $\varphi$ on a subset of $\mathcal A(n)$.   

If $(\lambda, \mu)\in \mathcal A(n)$, we write  
$$\lambda=\lambda^{0,2m}\cup \lambda^{\pm s,2m}\cup\lambda^{\pm(2m+2s), 4m},\hspace{.1in}\mu= \mu^{\pm s,2m}\cup \mu^{\pm(m+s),2m}$$ and we consider several cases.
 \medskip

\noindent \underline{Case 1:} $\lambda^{\pm s,2m}\cup \mu^{\pm s,2m}\neq \emptyset$ (i.e., there is at least one part congruent to $\pm s\pmod{2m}$ in $\lambda$ or in $\mu$). Define $\varphi(\lambda, \mu):=F(\lambda, \mu)$. 

Hence,  $\varphi$ is a sign reversing involution on the subset of pairs $(\lambda, \mu)$ in $\mathcal A(n)$ with $\lambda^{\pm s, 2m}\cup \mu^{\pm s, 2m}\neq \emptyset$. 
\medskip 

\noindent \underline{Case 2:} $\lambda^{\pm s,2m}\cup \mu^{\pm s,2m}= \emptyset$ and ($\lambda^{\pm(2m+2s),4m}\neq \emptyset$ or $\mu^{\pm(m+s),2m}$ has at least one repeated part). In this case  the involution is similar to that of \cite{G}. Let $a$ be the largest part in $\lambda^{\pm(2m+2s),4m}$ and $b$ be the largest repeated part in $\mu^{\pm(m+s),2m}$ (with $a=0$ or $b=0$ if no such part exists). 

(i) If $a<2b$, remove two parts equal to $b$ from $\mu^{\pm(m+s),2m}$ and insert a part equal to $2b$ into $\lambda^{\pm(2m+2s),4m}$, i.e., 
  $$\varphi(\lambda, \mu):=(\lambda\cup (2b), \mu\setminus  (b,b)).$$ 

  (i) If $a\geq 2b$, remove  part  $a $ from $\lambda^{\pm(2m+2s),4m}$ and insert two parts equal to $a/2$  into $\mu^{\pm(m+s),2m}$, i.e., 
  $$\varphi(\lambda, \mu):=(\lambda\setminus  (a), \mu\cup  (a/2, a/2)).$$ 
Again, $\varphi$ is a sign reversing involution on the subset of pairs in $\mathcal A(n)$ satisfying $\lambda^{\pm s,2m}\cup \mu^{\pm s,2m}= \emptyset$ and ($\lambda^{\pm(2m+2s),4m}\neq \emptyset$ or $\mu^{\pm(m+s),2m}$ has at least one repeated part). \medskip

\noindent \underline{Case 3:} $\lambda^{\pm s,2m}= \mu^{\pm s,2m}=\lambda^{\pm(2m+2s),4m}= \emptyset$ and  $\mu^{\pm(m+s),2m}$ is a  partition with distinct parts. Note that in this case $\lambda=\lambda^{0,2m}$ and $\mu=\mu^{\pm(m+s),2m}$. 
Thus, to complete the proof, we need a combinatorial proof of \eqref{eq_staircase}. Of course, using $2m$ in place of $m$ and $z=-q^m$ in the proof of \cite{KK} also provides a combinatorial proof of \eqref{eq_staircase}. However, the proof provided below is simpler.

We denote by  $\mathcal C(n)$ the subset of pairs $(\lambda, \mu)\in \mathcal A(n)$ satisfying the conditions of Case 3. We define a sign reversing involution on the set $$\mathcal C^*(n)=\mathcal C(n) \setminus  \left(\{(\emptyset, \delta_{m-s, 2m, k})\mid k\geq 0\} \cup\{(\emptyset, \delta_{m+s, 2m, |k|})\mid k<0\}\right).$$

    Now we use a particular case of \cite[Lemma 2.1]{BM}. 
    Following the notation in \cite{BM} we denote by $\mathcal Q_{2m,m-s}(n)$  the set of distinct partitions of $n$ with all parts congruent to $m\pm s\pmod{2m}$  and by $\mathcal W_{2m,m-s}(n)$  the set of pairs $(\nu, k(mk-s))$, where $\nu$ is a partition  into parts divisible by $2m$,  $k\in \mathbb Z$, and $|\nu|+ k(mk-s)=n$.
    Then, the map $\xi_{2m,m-s}: \mathcal Q_{2m,m-s}(n)\to \mathcal W_{2m,m-s}(n)$ defined in  \cite[Lemma 2.1]{BM}  is a bijection. 

Let $\mathcal D(n)$  be the set of triples $(\lambda, \nu, k(mk-s))$, where $\lambda$ and $\nu$ have parts divisible by $2m$, $\lambda$ has distinct parts, $k\in \mathbb Z$, and $|\lambda|+|\mu|+k(mk-s)=n$.  The map $\zeta:\mathcal C(n)\to \mathcal D(n)$ defined by $\zeta(\lambda,\mu)=(\lambda, \xi_{2m,m-s}(\mu))$ is a bijection. 

 We define a sign reversing involution $\psi$ on $\mathcal D(n)\setminus (\emptyset, \emptyset, k(mk-s))$.   Let $(\lambda, \nu, k(mk-s))\in \mathcal D(n)$ with $\lambda\cup \nu \neq \emptyset$ and define $\psi(\lambda, \nu, k(mk-s))=(F(\lambda, \nu),k(mk-s)).$
Then the proof of \cite[Lemma 2.1]{BM} shows that the map $\varphi(\lambda,\mu):=\zeta^{-1}\circ\psi\circ\zeta(\lambda, \mu)$ is a sign reversing involution on  $\mathcal C^{*}(n)$.

Hence, the left hand side of \eqref{I6.7-gen} is the generating function for $|\mathcal C(n)\setminus \mathcal C^*(n)|$. 

This completes the combinatorial proof of Theorem \ref{T-gen}.
\end{proof}

The combinatorial proof of Theorem \ref{T-gen} above also gives a combinatorial proof of \cite[Theorem 38]{KMR}.

\begin{theorem}[\cite{KMR}, Theorem 38] Let $a$ be $1$ or $3$. Furthermore, let $(a_n)_{n\geq 0}$ be the sequence of non-negative integers $j$ such that $16j+a^2$ is a square. Then \begin{equation}\label{Th38}\sum_{n=0}^\infty q^{a_n}=(q^8;q^8)_\infty(-q^{4+a};q^8)_\infty(-q^{4-a};q^8)_\infty.\end{equation}
    \end{theorem}
    \begin{proof}[Combinatorial Proof]
    First note that, if $a\in \{1,3\}$, then $16j+a^2$  is a square if and only if $j=k(4k-a)$  for some $k\in \mathbb Z$.  Then, Case 3 in the proof of Theorem \ref{T-gen} with $m=4$ and $s=a$ proves the theorem.
\end{proof}

\section{Distinct $5$-regular partitions.}\label{distinct-reg}

  Proving identities (6.8), (6.9), (6.11), and (6.14) in \cite{M1} combinatorially  is still an open problem. The difficulty consists in finding a combinatorial proof of the Weierstrass addition formula \cite[Identity (8.1)]{KMR} used in \cite{KMR} to prove these identities. 

In this section, we use the Weierstrass addition formula to prove Theorem \ref{d5} below, an interesting identity involving the number of distinct $5$-regular partitions of $n$. This is sequence 
A096938 in \cite{OEIS} where several other interpretations of this sequence are discussed. 

We denote by $\mathcal D_5(n)$ the set of distinct $5$-regular partitions of $n$, i.e., partitions of $n$ into distinct parts not divisible by $5$. If $a, b$ are integers such that $0<a,b<10$ and $a \not \equiv \pm b \pmod{10}$, we denote by  $\mathcal C_{a,b}(n)$  the set of pairs of partitions $(\lambda, \mu)\vdash n$ such that  $\lambda$ is a partition with distinct parts not congruent to $\pm a$ or $\pm b$ modulo $10$ and $\mu$ is a distinct partition with all parts divisible by $5$. 
\begin{theorem}\label{d5}
    For $n\geq 1$, $$|\mathcal D_5(n)|=| \mathcal C_{1,2}(n)|+| \mathcal C_{3,4}(n-1)|.$$
\end{theorem}

\begin{proof}
If $m\geq 2$ and $0<s<m$, we use the notation $$\theta(q^s;q^m):=(q^s;q^m)_\infty(q^{m-s};q^s)_\infty.$$ 

As explained in \cite[pg. 225]{KMR}, if we first replace $q$ by $q^{10}$ and then take $u=-q^3$, $v=q$, $x=q^4$ and $y=q^3$, in \cite[Identity (8.1)]{KMR}, we obtain
\begin{align}
   \notag \theta(q^2;q^{10})\theta(q^4;q^{10})& \theta(-q;q^{10})\theta(-q^3;q^{10})\\\notag  &  + \theta(q;q^{10})\theta(q^3;q^{10})\theta(-q^2;q^{10})\theta(-q^4;q^{10})\\ \label{WSF1} &\ \ \  = \theta(q^3;q^{10})\theta(-q^4;q^{10})\theta(q^5;q^{10})\theta(-1;q^{10}).
\end{align}

We give a combinatorial interpretation of \eqref{WSF1}.  The $q$-series $$\theta(q^2;q^{10})\theta(q^4;q^{10})\theta(-q;q^{10})\theta(-q^3;q^{10})$$ is the generating function for \begin{align*}& |\{\lambda \in \mathcal D_5(n) \mid \ell_e(\lambda) \text{ even}\}|-|\{\lambda \in \mathcal D_5(n) \mid \ell_e(\lambda) \text{ odd}\}|\\ & \ \ \ \ \ = (-1)^n\left(|\{\lambda \in \mathcal D_5(n) \mid \ell(\lambda) \text{ even}\}|-|\{\lambda \in \mathcal D_5(n) \mid \ell(\lambda) \text{ odd}\}|\right),\end{align*} and the $q$-series $$\theta(q;q^{10})\theta(q^3;q^{10})\theta(-q^2;q^{10})\theta(-q^4;q^{10})$$ is the generating function for $$|\{\lambda \in \mathcal D_5(n) \mid \ell_o(\lambda) \text{ even}\}|-|\{\lambda \in \mathcal D_5(n) \mid \ell_o(\lambda) \text{ odd}\}|=(-1)^n |\mathcal D_5(n)|.$$ Here, we used  the fact that $\ell_0(\lambda)\equiv |\lambda|\mod 2$.

 The $q$-series   \begin{align*}& \theta(q^3;q^{10})\theta(-q^4;q^{10})\theta(q^5;q^{10})\theta(-1;q^{10})\\ & =  2(q^3; q^{10})_\infty(q^7; q^{10})_\infty(-q^4; q^{10})_\infty(-q^6; q^{10})_\infty(q^5; q^{10})^2_\infty(-q^{10}; q^{10})^2_\infty\end{align*} is the generating function for \begin{align*}2 (|\{(\lambda, \mu)\in \mathcal C_{1,2}(n) \mid \ell_o(\lambda)+\ell_o(\mu) \text{ even}\}|&-|\{(\lambda, \mu)\in \mathcal C_{1,2}(n) \mid \ell_o(\lambda)+\ell_o(\mu) \text{ odd}\}|)\\ & = 2(-1)^n | \mathcal C_{1,2}(n)|.\end{align*}

Hence, the combinatorial interpretation of \eqref{WSF1}  is $$|\{\lambda \in \mathcal D_5(n) \mid \ell(\lambda) \text{ even}\}|=| \mathcal C_{1,2}(n)|.$$

If we now follow \cite[pg. 226]{KMR} and first replace $q$ by $q^{10}$ and then take $u=q^5$, $v=q^2$, $x=q^6$ and $y=q^4$, in \cite[Identity (8.1)]{KMR}, we obtain
\begin{align}
   \notag \theta(q^2;q^{10})\theta(q^4;q^{10})& \theta(-q;q^{10})\theta(-q^3;q^{10})\\\notag  &  - \theta(q;q^{10})\theta(q^3;q^{10})\theta(-q^2;q^{10})\theta(-q^4;q^{10})\\ \label{WSF2} &\ \ \  = q\theta(-q^2;q^{10})\theta(-1;q^{10})\theta(q^5;q^{10})\theta(q;q^{10}).
\end{align}
By the same argument as above, \begin{align*}\theta(-q^2;q^{10})\theta(-1;q^{10})\theta(q^5;q^{10})\theta(q;q^{10})\end{align*} is the generating function for $2(-1)^n|\mathcal C_{3,4}(n)|$ and the combinatorial interpretation of \eqref{WSF2}  is $$|\{\lambda \in \mathcal D_5(n) \mid \ell(\lambda) \text{ odd}\}|=| \mathcal C_{3,4}(n-1)|.$$ 
\end{proof}

 \section{Concluding remarks} \label{concluding}

 In this article we considered several conjectures and identities that first appeared in \cite{M1}. Identities (6.7), (6.10), (6.12) and (6.13) from \cite{M1} are proved using the Jacobi Triple Product identity in \cite{KMR}. We generalized these identities and gave both analytic and combinatorial proofs. Considerations of  similar identities from \cite{M1}, and their proofs in \cite{KMR} led us to discover an interesting result about the number of distinct $5$-regular partitions of $n$. Theorem  \ref{main} proves the cases $k=1,2,3$ of Conjecture 4.3 in \cite{M1} which concerns the non-negativity of the coefficients of a certain truncated theta function. The theorem gives inequalities involving the number of partitions of $n$ with parts congruent to $\pm S \mod R$ and it also leads to results involving the generalizations of the minimal excludant statistic for partitions. 

 At the end of \cite{M1}, Merca made another conjecture about the non-negativity of the coefficients of a  truncated series. Suppose $R\geq 2$ and $1\leq S<R$ and define  $$T_{\pm}(q,k,S,R):=\frac{1}{(q^S, q^{R-S};q^R)_\infty }\sum_{n=k}^\infty q^{n(3n\mp 1)R/2\pm 3nS}(1-q^{S\pm nR}).$$
Conjecture 6.15 on \cite{M1} states that, for all $k\geq 1$,  each of $T_{+}(q,k,S,R)$, $-T_{-}(q,k,S,R)$, and $T_{-}(q,k,S,R)+T_{+}(q,k,S,R)$ has non-negative coefficients.
We  establish the first two parts of the conjecture.

\begin{theorem} For $1\leq S<R$ and $k\geq 1$, \begin{itemize}
    \item[(i)]  $T_{+}(q,k,S,R)$ has non-negative coefficients;
    \item[(ii)]  $T_{-}(q,k,S,R)$ has non-positive coefficients.
\end{itemize}
\end{theorem}
\begin{proof}Let  $1\leq S <R$ and $k\geq 1$.

 To show that for the $q$-series $T_+(q,k,S,R)$ has non-negative coefficients, we rewrite it as 
\begin{align*}&T_+(q,k,S,R)
\\& =\frac{1}{(q^S, q^{R-S};q^R)_\infty} \sum_{n=k}^\infty q^{n(3n-1)R/2+3nS}(1-q^{n(R-S)}+q^{n(R-S)}-q^{S+nR})\\ & = \frac{1}{(q^S, q^{2R-S};q^R)_\infty} \sum_{n=k}^\infty q^{n(3n-1)R/2+3nS}\, \frac{1-q^{n(R-S)}}{1-q^{R-S}}\\ & \ \ \ + \frac{1}{(q^{S+R}, q^{R-S};q^R)_\infty} \sum_{n=k}^\infty q^{n(3n-1)R/2+3nS+nR-nS}\, \frac{1-q^{(n+1)S}}{1-q^S}\\ & = \frac{1}{(q^S, q^{2R-S};q^R)_\infty} \sum_{n=k}^\infty q^{n(3n-1)R/2+3nS} (1+q^{R-S}+q^{2(R-S)}+\cdots +q^{(n-1)(R-S)})\\ & \ \ \ + \frac{1}{(q^{S+R}, q^{R-S};q^R)_\infty} \sum_{n=k}^\infty q^{n(3n+1)R/2+2nS}(1+q^{S}+q^{2S}+\cdots+ q^{nS})
.\end{align*} Clearly, this $q$-series has non-negative coefficients. \medskip

 If $k\geq 2$, we rewrite
\begin{align*}&-T_-(q,k,S,R)\\ & = - \frac{1}{(q^S, q^{R-S};q^R)_\infty} \sum_{n=k}^\infty q^{n(3n+ 1)R/2- 3nS}(1-q^{S- nR})\\& =\frac{1}{(q^S, q^{R-S};q^R)_\infty} \sum_{n=k}^\infty q^{n(3n+1)R/2-3nS+S-nR}(1-q^{n(R-S)}+q^{n(R-S)}-q^{nR-S})\\ & = \frac{1}{(q^S, q^{2R-S};q^R)_\infty} \sum_{n=k}^\infty q^{n(3n+1)R/2-3nS+S-nR}\, \frac{1-q^{n(R-S)}}{1-q^{R-S}}\\ & \ \ \ + \frac{1}{(q^{S+R}, q^{R-S};q^R)_\infty} \sum_{n=k}^\infty q^{n(3n+1)R/2-3nS+S-nR+nR-nS}\, \frac{1-q^{(n-1)S}}{1-q^S}\\ & = \frac{1}{(q^S, q^{2R-S};q^R)_\infty} \sum_{n=k}^\infty q^{n(3n-1)R/2-3nS+S} (1+q^{R-S}+q^{2(R-S)}+\cdots +q^{(n-1)(R-S)})\\ & \ \ \ + \frac{1}{(q^{S+R}, q^{R-S};q^R)_\infty} \sum_{n=k}^\infty q^{n(3n+1)R/2-4nS+S}(1+q^{S}+q^{2S}+\cdots+ q^{(n-2)S})
.\end{align*}
If $k=1$, we have \begin{align*}&-T_-(q,1,S,R) =\frac{q^{R-2S}}{(q^S, q^{2R-S};q^R)_\infty}  - T_-(q,2,S,R). \end{align*} Clearly, in each case  $-T_-(q,1,S,R)$ has non-negative coefficients.
\end{proof} 

Finally, we note that results of a similar flavor appear in \cite{M2}.
 Let $g(n)=|\mathcal P_{\pm 1,5}(n)|$ (respectively $h(n)=|\mathcal P_{\pm 2,5}(n)|$) be the number of partitions of $n$ into parts congruent to $\pm 1 \mod 5$ (respectively $\pm 2\mod 5$). If $\xi(n)$ is either of $g(n)$ or $h(n)$, Merca \cite[Corollary 3.3]{M2} proves the following families of linear homogeneous inequalities. For $k\geq 1$, $S\in\{1,2\}$, and $n\geq 0$ \begin{equation}\label{eq:M2}
        u_{\xi,S}^{\pm}(n):= \pm \sum_{j=k}^\infty\left(\xi\left(n-5j(3j\pm 1)/2\pm 3jS\right)-\xi\left(n-5j(3j\pm 1)/2\mp (3j\pm 1)S\right) \right)\geq 0
    \end{equation} and asks for combinatorial interpretations of these sums. The requested interpretations can be read off the generating functions obtained in the proofs in \cite{M2}.

  In  the proof of \cite[Theorem 2.1]{M2} (with $R=5, S=1$) it is shown that $$\sum_{n=0}^\infty u_{g,1}^{+}(n)q^n=\frac{1}{(q^{6},q^{4};q^5)_\infty}\sum_{j=k}^\infty q^{5j(3j+1)/2-3j}\big(1+q+q^{2}+\cdots +q^{6j}\big).$$

    Now, $\displaystyle \frac{1}{(q^{6},q^{4};q^5)_\infty}$ is the generating function for the number of partitions $\lambda\in \mathcal P_{\pm 1,5}(n)$ of $n$ with no part equal to $1$. 
    We can interpret $\displaystyle q^{5j(3j+1)/2-3j}$ as generating a single partition  
$(5j+1)^j\cup \delta_{1,5,j}$.

    Then  $$u_{g,1}^+(n)=|\{(\lambda,(5j+1)^j\cup \delta_{1,5,j})\vdash n  \mid   \lambda\in \mathcal P_{\pm 1, 5}, m_\lambda(1)\leq 6j, j\geq k\}|.$$

In a similar way we can show that
\begin{align*}u_{h,2}^+(n)&=|\{ (\lambda, (5(j-1)+2)^j\cup \delta_{2,5,j})\vdash n \mid \lambda \in \mathcal P_{\pm2,5}, m_{\lambda}(2)\leq 6j, j\geq k \}|\\
u_{g,1}^-(n)&=|\{ (\lambda, (5j+1)^{j-1}\cup\delta_{1,5,j}\cup (4)) \vdash n \mid \lambda \in \mathcal P_{\pm1,5}, m_{\lambda}(1)\leq 6j-2, j\geq k\}|\\ 
u_{h,2}^-(n)&=|\{(\lambda, (5(j-1)+2)^{j-1}\cup \delta_{2,5,j}\cup (3)) \vdash n\mid \lambda \in \mathcal P_{\pm 2, 5}, m_\lambda(2)\leq 6j-2, j\geq k \}|. 
\end{align*}

Using the generating functions in  \cite[Corollary 3.2]{M2}, we have 
\begin{align*}u_{g,2}^+(n)&=|\{ (\lambda, (\delta_{4,5,2j-1}\setminus \delta_{4,5,j-1}))\vdash n \mid \lambda \in \mathcal P_{\pm1,5}, m_{\lambda}(1)\leq 12j+1, j\geq k \}|\\ u_{g,2}^-(n)&=|\{ (\lambda, (\delta_{4,5,2j-2}\setminus \delta_{4,5,j-1})\cup(5(j-1)+1))\vdash n \mid \lambda \in \mathcal P_{\pm1,5}, m_{\lambda}(1)\leq 12j-3, j\geq k \}|\\
u_{h,1}^+(n)&=|\{ (\lambda, (\delta_{2,5,2j}\setminus \delta_{2,5,j}))\vdash n \mid \lambda \in \mathcal P_{\pm2,5}, m_{\lambda}(2)\leq 3j-2, j\geq k \}|\\ & \ \ \ +|\{ (\lambda, (\delta_{3,5,2j+1}\setminus \delta_{3,5,j+2})\cup(5j+3))\vdash n \mid \lambda \in \mathcal P_{\pm2,5}, m_{\lambda}(3)= 1, j\geq k \}|
\\ u_{h,1}^-(n)&=|\{ (\lambda, (\delta_{2,5,2j-1}\setminus \delta_{2,5,j}) \cup(5(j-1)+3))\vdash n \mid \lambda \in \mathcal P_{\pm2,5}, m_{\lambda}(2)\leq 3j-3, j\geq k \}|\\ & \ \ \ 
+|\{ (\lambda, (\delta_{3,5,2j}\setminus \delta_{3,5,j+1})\cup(5(j-1)+2))\vdash n \mid \lambda \in \mathcal P_{\pm2,5}, m_{\lambda}(3)= 1, j\geq k \}|.
\end{align*}


\begin{thebibliography}{00}



 
\bibitem{A98} 
G. E. Andrews, 
\textit{The Theory of Partitions}, 
Cambridge Mathematical Library, Cambridge University Press, Cambridge, 1998. Reprint of the 1976 original.

\bibitem{A72} G. E. Andrews,  \textit{Two theorems of Gauss and allied identities proved arithmetically.} Pacific J. Math. 41 (1972), 563--578.

\bibitem{A86} G. E. Andrews,  \textit{The fifth and seventh order mock theta functions}, Trans. Amer. Math. Soc. 293 (1986), no. 1, 113--134
 
\bibitem{AE} 
G. E. Andrews,  K. Eriksson,  \textit{Integer partitions}, Cambridge University Press, Cambridge, 2004. x+141 pp. ISBN: 0-521-84118-6; 0-521-60090-1

\bibitem{AM} G. E. Andrews, M.  Merca,  \textit{The truncated pentagonal number theorem}, J. Combin. Theory Ser. A 119 (2012), no. 8, 1639--1643.

\bibitem{AN} G. E. Andrews, D. Newman,  \textit{The minimal excludant in integer partitions}, J. Integer Seq. 23 (2020), no. 2, Art. 20.2.3, 11 pp.



\bibitem{BBCFW}  C. Ballantine, H. Burson, W. Craig, A. Folsom, B. Wen, 	\textit{Hook length biases and general linear partition inequalities}, Res. Math. Sci. 10 (2023), no. 4, Paper No. 41, 36 pp.


\bibitem{BM} C. Ballantine, M. Merca, \textit{Almost $3$-regular overpartitions}, Ramanujan J. {58} (2022), no. 3, 957–971.



\bibitem{BM3} C. Ballantine, M. Merca, \textit{$6$-regular partitions: new combinatorial properties, congruences, and linear inequalities}, to appear in  Rev. R. Acad. Cienc. Exactas Fís. Nat. Ser. A Mat. RACSAM.



\bibitem{GZ} V. J. W. Guo, J. Zeng,  \textit{Two truncated identities of Gauss}, J. Combin. Theory Ser. A, 120 (2013), 700--707.

\bibitem{G}
H.\ Gupta,
\textit{Combinatorial proof of a theorem on partitions into an even or odd number of parts}. J. Combinatorial Theory Ser. A 21 (1976), no. 1, 100–103.


\bibitem{KK} L. W. Kolitsch and  S. Kolitsch, \textit{
A combinatorial proof of Jacobi's triple product identity},
Ramanujan J. 45 (2018), no. 2, 483--489.

\bibitem{KMR} C. Krattenthaler, M. Merca, C.-S. Radu, \textit{Infinite product formulae for generating functions for sequences of squares}, in  \textit{Transcendence in algebra, combinatorics, geometry and number theory,} 193–236, Springer Proc. Math. Stat., 373, Springer, Cham, [2021], ©2021. 
 
 
 \bibitem{M1} M. Merca,  \textit{Truncated theta series and Rogers-Ramanujan functions}, {Exp. Math.} {30} (2021), no. 3, 364--371.

 \bibitem{M2} M. Merca,  \textit{On two truncated quintuple series theorems}, {Exp. Math.} 31 (2022), no. 2, 606--610.

 \bibitem{M3} M. Merca,  \textit{Linear inequalities concerning partitions into distinct parts}, Ramanujan J. 58 (2022), no. 2, 491--503.

 \bibitem{M4} M. Merca,  \textit{Rank partition functions and truncated theta identities}, Appl. Anal. Discrete Math. 
https://doi.org/10.2298/AADM190401023M



 \bibitem{S1} D.  Shanks,  \textit{A short proof of an identity of Euler}, Proc. Amer. Math. Soc. 2 (1951), 747--749.
 
 \bibitem{S2} D.  Shanks,  \textit{Two theorems of Gauss}, Pacific J. Math. 8 (1958), 609--612.

\bibitem{OEIS} N. J. A. Sloane, editor, The On-Line Encyclopedia of Integer Sequences. Published electronically at oeis.org, 2023.
 
 \bibitem{WY} C. Wang, A. J. Yee, \textit{Truncated Jacobi triple product series}, J. Comb. Th., Ser. A, 166 (2019), 382--392.

 \bibitem{XZ} E. X. Xia, X.  Zhao, \textit{Truncated sums for the partition function and a problem of Merca}, Rev. R. Acad. Cienc. Exactas Fís. Nat. Ser. A Mat. RACSAM 116 (2022), no. 1, Paper No. 22, 8 pp.
 
\bibitem{Y}  A. J. Yee, \textit{Truncated Jacobi triple product theorems}, J. Comb. Th., Ser. A, 130 (2015), 1--14.
 
\end{thebibliography}
\end{document}